\documentclass[11pt]{article}

\usepackage[T1]{fontenc}
\usepackage{lmodern}
\usepackage{microtype}
\usepackage[margin=1in]{geometry}
\usepackage{amsmath,amssymb,amsthm,mathtools}
\usepackage[hidelinks]{hyperref}

\allowdisplaybreaks

\newtheorem{theorem}{Theorem}[section]
\newtheorem{proposition}[theorem]{Proposition}
\newtheorem{lemma}[theorem]{Lemma}
\newtheorem{corollary}[theorem]{Corollary}
\theoremstyle{remark}

\newcommand{\Mat}{\operatorname{Mat}}
\newcommand{\tr}{\operatorname{tr}}
\newcommand{\Res}{\operatorname{Res}}
\newcommand{\rat}{\mathrm{rat}}
\newcommand{\id}{\mathrm{id}}
\newcommand{\Q}{\mathbb{Q}}

\title{A Matrix-Degree Obstruction to Rational Generation\\
of Boolean-Lattice Pseudo-Roots}
\author{Mahesh Ramani\thanks{Email: \href{mailto:mahesh.ramani.iyer@gmail.com}{\nolinkurl{mahesh.ramani.iyer@gmail.com}}}}
\date{}

\hypersetup{
  pdftitle={A Matrix-Degree Obstruction to Rational Generation of Boolean-Lattice Pseudo-Roots},
  pdfauthor={Mahesh Ramani},
  pdfsubject={Noncommutative algebra, free skew fields, and rational maps},
  pdfkeywords={pseudo-roots, free skew field, noncommutative rational functions, rational maps, matrix degree, cographs}
}

\begin{document}

\maketitle

\begin{abstract}
For the neighborhood seed associated with the four-vertex path $P_4$, the
diamond operations do not recover all Boolean-lattice pseudo-roots.  The
corresponding question for unrestricted rational operations in the free skew
field is subtler: the seed map has an invertible linearization and therefore a
unique formal inverse near every generic scalar point.  We prove that this
formal inverse is not free rational.  A symmetric one-parameter curve of
$2\times2$ matrix outputs has a formal inverse whose coefficient field contains
an element of degree three over $\Q(t)$.  An exact elimination in a quadratic
Pauli algebra produces the irreducible cubic.  Its conjugate inverse branches
are unramified, forcing the generic matrix degree of the seed map to be at least
three in every size $n\geq2$.  This contradicts the degree-one consequence of
any free rational inverse.  The same matrix-degree argument, without
specializing a hypothetical inverse, extends the obstruction to every graph
containing an induced $P_4$.
\end{abstract}

\medskip
\noindent\textbf{2020 Mathematics Subject Classification.}
Primary 16K40; Secondary 16S10, 16S85, 14E05, 05C75.

\smallskip
\noindent\textbf{Keywords.}
Pseudo-roots; free skew fields; noncommutative rational functions; rational
maps; matrix degree; cographs; exact elimination.

\section{Introduction}

Let $B_n$ be the Boolean lattice on $[n]=\{1,\ldots,n\}$.  Its cover edge
$A\to A\cup\{i\}$ is labelled by a pseudo-root $x_{A,i}$.  This
Boolean-lattice model and its diamond relations arise from factorizations of
noncommutative polynomials; see
\cite{GelfandRetakhWilson2001,GelfandEtAl2005,RetakhSaks2021}.  If two lower
pseudo-roots are denoted by $x$ and $y$, the upper pseudo-roots in the
corresponding diamond are $C(y,x)$ and $C(x,y)$, where
\begin{equation}\label{eq:CD}
 C(y,x)=(y-x)x(y-x)^{-1},
 \qquad
 D(y,x)=(y-x)^{-1}x(y-x).
\end{equation}
Here $D$ denotes the reverse diamond operation, used below to invert a diamond.
All free rational expressions below are interpreted in the free skew field,
or on matrix tuples for which every displayed inverse exists.  We use the
standard universal-field and matrix-evaluation framework for noncommutative
rational functions \cite{Cohn1985,KaliuzhnyiVerbovetskyiVinnikov2014}.

For a graph $G$ on $[n]$, its neighborhood seed consists of the $n$
pseudo-roots
\[
  Y_G=\{x_{N_G(i),i}:i\in[n]\}.
\]
At the level of diamond closure, the cographs are precisely the graphs for
which $Y_G$ generates all pseudo-roots \cite[Theorem~1.1]{Ramani2026}.
Cographs are equivalently the finite graphs with no induced $P_4$
\cite{CorneilLerchsBurlingham1981}.  The first graph outside this class is the
path $P_4$.  Write its bottom pseudo-roots as $a,b,c,d$.  Its seed map is
\begin{equation}\label{eq:FP4intro}
 F_{P_4}(a,b,c,d)
 =\bigl(C(b,a),\ C(C(a,c),C(a,b)),\
 C(C(b,d),C(b,c)),\ C(c,d)\bigr).
\end{equation}

The main result is the following.

\begin{theorem}\label{thm:main}
Let $K$ be a field of characteristic zero.  If
\begin{align*}
 p&=C(b,a),
 &q&=C(C(a,c),C(a,b)),\\
 r&=C(C(b,d),C(b,c)),
 &s&=C(c,d),
\end{align*}
then
\[
 K\langle p,q,r,s\rangle_{\rat}
 \subsetneq
 K\langle a,b,c,d\rangle_{\rat}.
\]
Equivalently, $F_{P_4}$ has no free rational inverse.
\end{theorem}

The proof uses matrix evaluations, but not merely a failure at one fixed
matrix size.  For every $n\geq2$, the rational map induced by $F_{P_4}$ on four
$n\times n$ matrices has generic degree at least three.  A finite collection of
free rational recovery expressions necessarily has a common matrix domain at
some size, where it would give a rational left inverse and force generic degree
one.  This all-size argument avoids the polynomial-identity loophole inherent
in a single small-matrix specialization.

The only computer-assisted step is an exact resultant calculation.  The
ancillary file \nolinkurl{p4_matrix_degree_certificate.py} performs the
calculation over $\Q$, verifies the formal branch used to select the relevant
factors, and checks the irreducibility specialization.  No floating-point
arithmetic or numerical root selection is used.

\section{Diamond inversion and the \texorpdfstring{$P_4$}{P4} reductions}

\begin{lemma}[Diamond inversion]\label{lem:diamond-inversion}
Let
\[
 u=C(y,x),\qquad v=C(x,y).
\]
Then
\[
 v-u=y-x,\qquad x=D(v,u),\qquad y=D(u,v).
\]
Consequently, the rational transformations
\[
 (x,y)\longmapsto (C(y,x),C(x,y))
 \quad\text{and}\quad
 (u,v)\longmapsto(D(v,u),D(u,v))
\]
are mutually inverse wherever their displayed inverses exist.
\end{lemma}

\begin{proof}
Since $x-y=-(y-x)$,
\[
 v=(y-x)y(y-x)^{-1},
 \qquad
 u=(y-x)x(y-x)^{-1}.
\]
Thus $v-u=y-x$.  It follows that
\[
 D(v,u)=(v-u)^{-1}u(v-u)
 =(y-x)^{-1}(y-x)x(y-x)^{-1}(y-x)=x.
\]
The proof of $D(u,v)=y$ is identical.  Substituting these formulas back into
$C$ proves the final assertion.
\end{proof}

For $P_4=1-2-3-4$, set
\[
 a=x_{\varnothing,1},\quad b=x_{\varnothing,2},\quad
 c=x_{\varnothing,3},\quad d=x_{\varnothing,4}.
\]
The four neighborhood seeds are
\begin{equation}\label{eq:seeds}
\begin{aligned}
 p&=x_{\{2\},1}=C(b,a),\\
 q&=x_{\{1,3\},2}=C(C(a,c),C(a,b)),\\
 r&=x_{\{2,4\},3}=C(C(b,d),C(b,c)),\\
 s&=x_{\{3\},4}=C(c,d).
\end{aligned}
\end{equation}
Define the missing adjacent partners
\begin{equation}\label{eq:TW}
 T=C(a,b),\qquad W=C(d,c).
\end{equation}

\begin{proposition}[Partner reduction]\label{prop:partner}
The map $F_{P_4}$ is free rationally invertible if and only if $T$ and $W$ are
simultaneously free rational functions of $p,q,r,s$.
\end{proposition}

\begin{proof}
If $T$ and $W$ are recoverable, Lemma~\ref{lem:diamond-inversion} gives
\[
 a=D(T,p),\qquad b=D(p,T),\qquad
 d=D(W,s),\qquad c=D(s,W).
\]
Conversely, recovery of $a,b,c,d$ gives $T=C(a,b)$ and $W=C(d,c)$.
\end{proof}

There is also a useful conjugator form of this reduction.  Put
\[
 u=b-a,\qquad v=c-d.
\]
The relation in Lemma~\ref{lem:diamond-inversion} gives
\begin{equation}\label{eq:conjugators}
 T-p=u,\qquad W-s=v,\qquad
 a=u^{-1}pu,\quad b=u^{-1}pu+u,\quad
 d=v^{-1}sv,\quad c=v^{-1}sv+v.
\end{equation}
Hence
\[
 K\langle a,b,c,d\rangle_{\rat}
 =K\langle p,s,u,v\rangle_{\rat},
\]
and partner recovery is equivalent to recovery of $u$ and $v$.

For completeness, define the second-layer partners
\[
 M=x_{\{1,2\},3},\qquad N=x_{\{3,4\},2}.
\]
The relevant diamonds give
\[
 T=D(M,q),\qquad W=D(N,r),
\]
so Proposition~\ref{prop:partner} is equivalently a recovery problem for
$M,N$.  An exact two-variable consistency system is
\begin{align}
 D(q,M)
 &=C\!\left(D(D(M,q),p),D(s,D(N,r))\right),\label{eq:consistency-M}\\
 D(r,N)
 &=C\!\left(D(D(N,r),s),D(p,D(M,q))\right).\label{eq:consistency-N}
\end{align}
For the true partners, the left sides are respectively $C(a,c)$ and
$C(d,b)$, so the equations hold.  Conversely, suppose $M,N$ solve
\eqref{eq:consistency-M}--\eqref{eq:consistency-N}.  Define
\[
 T=D(M,q),\quad W=D(N,r),\quad
 a=D(T,p),\quad b=D(p,T),\quad
 d=D(W,s),\quad c=D(s,W).
\]
The adjacent pairs reconstruct $p,s$ by Lemma~\ref{lem:diamond-inversion}.
The two consistency equations say
\[
 D(q,M)=C(a,c),\qquad D(r,N)=C(d,b).
\]
Applying the mutually inverse diamond transformations to the upper pairs
$(q,M)$ and $(r,N)$ now reconstructs the middle seeds $q,r$.  Thus every
solution for which the displayed inverses exist gives a preimage of
$(p,q,r,s)$.

\section{Formal invertibility and matrix degree}

We record the distinction between formal and rational inversion.  Let
$\lambda_1,\ldots,\lambda_g$ be distinct central scalars and write a matrix or
free variable as $x_i=\lambda_i+X_i$, where $X_i$ has positive degree.

\begin{lemma}[First-order conjugation]\label{lem:first-order}
If $x=\lambda+X$, $y=\mu+Y$, and $\lambda\ne\mu$ are central scalars, then
\[
 C(y,x)=\lambda+X+O(2),
 \qquad
 D(y,x)=\lambda+X+O(2).
\]
\end{lemma}

\begin{proof}
Put $\delta=\mu-\lambda$ and $E=Y-X$.  Then
\[
 (y-x)^{-1}=(\delta+E)^{-1}
 =\delta^{-1}-\delta^{-2}E+O(2).
\]
Substitution into $(y-x)x(y-x)^{-1}$ shows that the degree-zero and degree-one
terms are $\lambda$ and $X$; the terms involving $\lambda E$ cancel because
$\lambda$ and $\delta$ are central.  The calculation for $D$ is the same.
\end{proof}

\begin{proposition}[Formal inverse]\label{prop:formal}
At every scalar point
\[
 (a,b,c,d)=(\lambda_1,\lambda_2,\lambda_3,\lambda_4)
\]
with distinct $\lambda_i$, the linear part of $F_{P_4}$ is the identity.
Consequently, $F_{P_4}$ has a unique inverse in the completed noncommutative
power-series algebra at that point.  For every matrix size $n$, the induced
rational map $F_{P_4,n}$ has identity Jacobian there and is dominant and
generically finite.
\end{proposition}

\begin{proof}
Every use of $C$ retains, to first order, the perturbation of its second
argument by Lemma~\ref{lem:first-order}.  Reading the four expressions in
\eqref{eq:seeds} from the inside out gives
\[
 p=\lambda_1+A+O(2),\quad
 q=\lambda_2+B+O(2),\quad
 r=\lambda_3+G+O(2),\quad
 s=\lambda_4+D+O(2).
\]
An endomorphism of a complete filtered algebra with invertible linear part has
a unique formal inverse, obtained recursively by degree.  The same
linearization on $n\times n$ matrix entries is the identity map on an affine
space of dimension $4n^2$.  A rational self-map with a nonsingular Jacobian at
one point is dominant; equal source and target dimensions then imply generic
finiteness.
\end{proof}

We shall use two elementary degree principles.  Degrees are geometric degrees,
computed after extension to an algebraic closure.  Our conventions for
rational maps, function fields, and birationality are standard; see
\cite[Chapter~I, \S4]{Hartshorne1977}.

\begin{lemma}[Simple points bound the generic degree]\label{lem:simple-fiber}
Let $f:X\dashrightarrow Y$ be a dominant generically finite rational map
between irreducible smooth varieties of the same dimension over a
characteristic-zero field.  If one geometric fiber contains $r$ distinct
points in the domain of $f$ at which the Jacobian determinant is nonzero, then
\[
 \deg(f)\ge r.
\]
\end{lemma}

\begin{proof}
All data involved in the points and the map are defined over a finitely
generated characteristic-zero field, which embeds in $\mathbb C$.  After this
embedding, the complex inverse function theorem gives pairwise disjoint
analytic neighborhoods of the $r$ points and a common analytic neighborhood
of their image on which each neighborhood supplies a distinct local inverse
branch.  A dominant generically finite map restricts, after deleting proper
algebraic subsets of source and target, to a finite \'{e}tale map of degree
$\deg(f)$.  The deleted subsets and their inverse images under the local
branches have empty analytic interior.  A nearby point can therefore be chosen
in the finite \'{e}tale locus and on every local branch.  Its fiber has at least
$r$ points but exactly $\deg(f)$ points.  Hence $\deg(f)\ge r$.
\end{proof}

\begin{lemma}[Free left inverses force matrix degree one]\label{lem:free-degree}
Let $F$ be a free rational self-map in $g$ variables.  Suppose every matrix
evaluation $F_n$ is dominant.  If $F$ has a free rational left inverse, then
$F_n$ is birational for at least one $n\ge2$.
\end{lemma}

\begin{proof}
Choose rational expressions for the finitely many components of the proposed
left inverse.  Each expression has a nonempty matrix domain at some size.
Domains are stable under direct sums, so after taking a common multiple of
these sizes, enlarged if necessary to be at least two, every component has a
nonempty domain at one common size $n$.  At fixed size, the domain of a rational
expression is Zariski open.  Since the matrix tuple space is irreducible, the
finitely many nonempty domains have nonempty intersection.

Thus the left inverse determines a rational map $R_n$ on a nonempty open subset
of the target matrix space.  Dominance of $F_n$ ensures that the composition
$R_n\circ F_n$ is defined on a nonempty open subset.  Equality of the free
rational functions gives
\[
 R_n\circ F_n=\id
\]
as rational maps.  On function fields this becomes
\[
 F_n^*\circ R_n^*=\id.
\]
The injective homomorphism $F_n^*$ is therefore also surjective, so it is an
isomorphism.  Hence $F_n$ is birational.
\end{proof}

\section{A symmetric \texorpdfstring{$2\times2$}{2-by-2} output curve}

Let
\begin{equation}\label{eq:XYZ}
 X=\begin{pmatrix}0&1\\1&0\end{pmatrix},\qquad
 Y=\begin{pmatrix}0&1\\-1&0\end{pmatrix},\qquad
 Z=\begin{pmatrix}1&0\\0&-1\end{pmatrix}.
\end{equation}
They satisfy
\[
 X^2=Z^2=I,\qquad Y^2=-I,
\]
and anticommute pairwise.  Consider the output curve
\begin{equation}\label{eq:curve}
 p=tX,\qquad q=I+tX,\qquad r=2I+tZ,\qquad s=3I+tZ.
\end{equation}
At $t=0$ this is the scalar point $(0,I,2I,3I)$.  By
Proposition~\ref{prop:formal}, it has a unique formal preimage
\begin{equation}\label{eq:preimage}
 a=A,\qquad b=I+B,\qquad c=2I+G,\qquad d=3I+D
\end{equation}
in $\Mat_2(\Q[[t]])^4$.

The symmetry of this branch reduces the elimination to three scalar
coordinates.  Set
\[
 H=\begin{pmatrix}1&1\\1&-1\end{pmatrix},
 \qquad
 \kappa(M)=(HMH^{-1})^{\mathsf T}.
\]
Since $H^{\mathsf T}=H$ and $H^2=2I$, the map $\kappa$ is an anti-involution.
Direct calculation gives
\begin{equation}\label{eq:kappa-basis}
 \kappa(X)=Z,\qquad \kappa(Y)=Y,\qquad \kappa(Z)=X.
\end{equation}

For traceless $2\times2$ matrices $U,V$, define
\[
 \Gamma_\delta(V,U)
 =(\delta I+V-U)U(\delta I+V-U)^{-1}.
\]
This is the centered part of $C(\mu I+V,\lambda I+U)$ when
$\delta=\mu-\lambda$.

\begin{lemma}[Reversal identities]\label{lem:reversal}
For traceless $2\times2$ matrices,
\begin{equation}\label{eq:kappa-gamma}
 \kappa\!\left(\Gamma_\delta(V,U)\right)
 =\Gamma_{-\delta}(\kappa V,\kappa U).
\end{equation}
Moreover, whenever the expressions are defined,
\begin{equation}\label{eq:consistency}
 C(C(x,z),C(x,y))=C(C(z,x),C(z,y)).
\end{equation}
\end{lemma}

\begin{proof}
If $R$ is traceless, then $R^2$ is scalar and
$(\delta I+R)^{-1}$ is a scalar multiple of $\delta I-R$.  Applying the
anti-involution to the defining conjugation for $\Gamma_\delta$ therefore
replaces its conjugator, up to an irrelevant nonzero scalar, by
$-\delta I+\kappa V-\kappa U$.  This proves \eqref{eq:kappa-gamma}.

For \eqref{eq:consistency}, put $A=x-y$ and $B=z-y$.  The elementary identity
\[
 B(A-B)^{-1}A=A(A-B)^{-1}B
\]
follows by substituting $A=(A-B)+B$ into both sides.  Substitution into the
definitions of $C$ then gives
\[
 (C(x,z)-C(x,y))(x-y)
 =(C(z,x)-C(z,y))(z-y).
\]
Both sides of \eqref{eq:consistency} consequently conjugate $y$ by this same
element.
\end{proof}

Let $\widehat F$ be the centered form of $F_{P_4}$, so that
\[
 F_{P_4}(A,I+B,2I+G,3I+D)
 =(P,I+Q,2I+R,3I+S).
\]
On traceless tuples define
\[
 \tau(A,B,G,D)=(\kappa D,\kappa G,\kappa B,\kappa A).
\]
Equations \eqref{eq:kappa-gamma} and \eqref{eq:consistency}, applied from the
inner diamonds outward, give
\begin{equation}\label{eq:equivariance}
 \widehat F\circ\tau=\tau\circ\widehat F.
\end{equation}
The first and fourth coordinates follow directly from
\eqref{eq:kappa-gamma}; equation \eqref{eq:consistency} changes the common
lower vertex in each middle coordinate before \eqref{eq:kappa-gamma} is
applied, interchanging the second and third coordinates.

Every centered matrix in \eqref{eq:preimage} is traceless.  Indeed, $p,q,r,s$
are respectively conjugate to $a,b,c,d$, and the output curve has scalar parts
$0,1,2,3$ and traceless perturbations.  The centered output tuple in
\eqref{eq:curve} is $(tX,tX,tZ,tZ)$, which is fixed by $\tau$.  Formal
uniqueness and \eqref{eq:equivariance} therefore imply
\begin{equation}\label{eq:GDsymmetry}
 G=\kappa(B),\qquad D=\kappa(A).
\end{equation}

Recursive substitution through degree three gives
\begin{equation}\label{eq:formal-jet}
\begin{aligned}
 A&=tX+4t^3Z+O(t^4),\\
 B&=tX-2t^2Y+(2X+4Z)t^3+O(t^4),\\
 G&=tZ-2t^2Y+(4X+2Z)t^3+O(t^4),\\
 D&=tZ+4t^3X+O(t^4).
\end{aligned}
\end{equation}
These identities are also checked directly by the ancillary exact-arithmetic
script.

\section{Exact elimination on the formal branch}

Because $q$ is conjugate to $b$, equation \eqref{eq:curve} gives
\[
 \tr(B)=0,\qquad B^2=t^2I.
\]
The matrices $X,Y,Z$ form a basis of the traceless $2\times2$ matrices, so
there are unique $x,y,z\in\Q[[t]]$ such that
\begin{equation}\label{eq:Bxyz}
 B=xX+yY+zZ,\qquad x^2-y^2+z^2=t^2.
\end{equation}
By \eqref{eq:GDsymmetry},
\[
 G=zX+yY+xZ.
\]
The jet \eqref{eq:formal-jet} gives
\begin{equation}\label{eq:xyz-jet}
 x=t+2t^3+O(t^4),\qquad
 y=-2t^2+O(t^4),\qquad
 z=4t^3+O(t^4).
\end{equation}

The equation $p=C(b,a)=tX$ determines $A$ rationally from $B$.  Since $A$ is
conjugate to $tX$, we have $A^2=t^2I$.  The relation
$p(b-a)=(b-a)a$ then implies
\[
 p(b+p)-(b+p)a=p^2-a^2=0.
\]
Thus, with
\[
 L=I+B+tX,\qquad e=2t^2+2tx-1,
\]
we have $A=L^{-1}(tX)L$.  Here
$\det L=1-2t^2-2tx=-e$, so $L$ is invertible on the formal branch.  Calculation
in the Pauli basis gives
\begin{equation}\label{eq:Aexplicit}
\begin{aligned}
 A={}&\frac{t(2x(t+x)-1)}{e}X
 +\frac{2t((t+x)y+z)}{e}Y\\
 &+\frac{2t((t+x)z+y)}{e}Z.
\end{aligned}
\end{equation}

Put
\[
 T=C(a,b),\qquad U=C(a,c),\qquad h=(U-T)(a-b).
\]
Since $T=(a-b)b(a-b)^{-1}$ and $q=(U-T)T(U-T)^{-1}$,
\[
 q=hbh^{-1}.
\]
Using $q=I+tX$ and $b=I+B$, this is equivalent to
\begin{equation}\label{eq:Ezero}
 E:=tXh-hB=0.
\end{equation}

Let $\mathcal A=\Q(t,x,z)(y)$ be the quadratic field determined by
\[
 y^2=x^2+z^2-t^2.
\]
The radicand has simple zeros when regarded as a rational function of $x$ over
$\Q(t,z)$, so it is not a square in $\Q(t,x,z)$.
Substitute \eqref{eq:Bxyz}, \eqref{eq:Aexplicit}, and
$G=zX+yY+xZ$ into \eqref{eq:Ezero}, calculating with
\[
 XY=-Z,\qquad XZ=-Y,\qquad YZ=-X
\]
and anticommutation.  Write the scalar and $X$-coordinates of $E$ as
\[
 E_0=a_0+b_0y,\qquad E_X=a_1+b_1y,
\]
where $a_i,b_i\in\Q(t,x,z)$.  On the branch \eqref{eq:xyz-jet}, exact
expansion gives
\begin{equation}\label{eq:b-leading}
 b_0=-t^2+O(t^4),
 \qquad
 b_1=t+O(t^3).
\end{equation}
In particular, both are nonzero in $\Q((t))$.  Therefore
\[
 y=\rho_0:=-\frac{a_0}{b_0}
 =\rho_1:=-\frac{a_1}{b_1}
\]
on the formal branch.

Cancel each $\rho_i$ in $\Q(t,x,z)$ before clearing denominators.  Let
$S(t,x,z)$ be the unique irreducible factor vanishing at $(0,0,0)$ in the
numerator of $\rho_0-\rho_1$, and let $N(t,x,z)$ be the corresponding factor
in the numerator of
\[
 \rho_0^2-(x^2+z^2-t^2).
\]
All other irreducible factors in these two numerators have nonzero value at the
origin and hence are units after substitution of the formal branch.  Thus
\begin{equation}\label{eq:SNzero}
 S(t,x(t),z(t))=N(t,x(t),z(t))=0.
\end{equation}

\begin{lemma}[Exact elimination]\label{lem:resultant}
The factors just defined have $x$-degrees $4$ and $8$, respectively, and
\begin{align}
 \Res_x(S,N)
 ={}&16t^{12}z(z-t)^2(t-1)^4(t+1)^4
 (2t-1)^3(2t+1)^3 \notag\\
 &\cdot(3t^2+tz-2)^2
 \left(8t^4+4t^3z+4t^2z^2-8t^2-4tz+1\right)^2
 P(t,z),
 \label{eq:resultant}
\end{align}
where
\[
 P(t,z)=A_3(t)z^3+A_2(t)z^2+A_1(t)z+A_0(t)
\]
and
\begin{align*}
 A_3(t)&=256t^{12}-1472t^{10}+2096t^8-676t^6+160t^4-25t^2,\\
 A_2(t)&=256t^{11}-1984t^9+1808t^7+28t^5+86t^3+20t,\\
 A_1(t)&=-2112t^{10}+1552t^8+1860t^6-257t^4+140t^2-4,\\
 A_0(t)&=-2880t^9+672t^7-564t^5+16t^3.
\end{align*}
\end{lemma}

\begin{proof}
This is a finite exact calculation in the quadratic algebra $\mathcal A$.
The ancillary script constructs $A,B,G,T,U,h,E$, extracts $S,N$ by exact
factorization over $\Q$, and verifies \eqref{eq:resultant} by expanding the
difference between its two sides.  It also verifies the nonzero leading terms
in \eqref{eq:b-leading}.  Every operation is symbolic over $\Q$; no numerical
approximation enters the calculation.
\end{proof}

From \eqref{eq:SNzero}, the left side of \eqref{eq:resultant} vanishes after
substitution of the formal branch.  In the integral domain $\Q[[t]]$, the
factors $t$, $z$, and $z-t$ are nonzero by \eqref{eq:xyz-jet}; every other
displayed factor except $P$ has nonzero constant term.  Hence
\begin{equation}\label{eq:Pzero}
 P(t,z(t))=0.
\end{equation}

\begin{lemma}[Irreducibility of the cubic]\label{lem:irreducible}
The polynomial $P(t,z)$ is irreducible in $\Q(t)[z]$.
\end{lemma}

\begin{proof}
The coefficient polynomials $A_3,A_2,A_1,A_0$ have greatest common divisor
one, so $P$ is primitive in $\Q[t][z]$.  At $t=2$,
\[
 P(2,z)=4(9255z^3-64618z^2-412473z-351616).
\]
The cubic in parentheses is primitive.  Modulo $7$ it becomes
\[
 z^3-z^2+2z+1,
\]
whose values at $0,1,\ldots,6$ are $1,3,2,4,1,6,4$.  It has no root in
$\mathbb F_7$ and is therefore irreducible there, hence irreducible over
$\Q$.  The leading coefficient $A_3(2)$ is nonzero.  If $P$ factored over
$\Q(t)$, Gauss's lemma would give a factorization in $\Q[t,z]$; specialization
at $t=2$ would preserve the positive $z$-degrees of both factors and would
factor $P(2,z)$, a contradiction.
\end{proof}

Combining \eqref{eq:Pzero} with Lemma~\ref{lem:irreducible} gives
\begin{equation}\label{eq:z-degree}
 [\Q(t,z(t)):\Q(t)]=3.
\end{equation}

\section{The matrix-degree obstruction}

Let $F_n$ be the rational map induced by $F_{P_4}$ on four $n\times n$
matrices.  By Proposition~\ref{prop:formal}, $F_n$ is dominant and generically
finite.  We now convert \eqref{eq:z-degree} into a lower bound for its degree.

\begin{proposition}\label{prop:all-degrees}
For every $n\ge2$,
\[
 \deg(F_n)\ge3.
\]
\end{proposition}

\begin{proof}
First let $n=2$.  Let $L$ be the field over $\Q(t)$ generated by all entries
of the formal preimage \eqref{eq:preimage}.  The equations defining the inverse
branch are rational algebraic equations whose denominators have nonzero
constant terms.  Their Jacobian with respect to the input entries is the
identity at $t=0$.  The algebraic implicit-function theorem therefore shows
that every entry of the branch is algebraic over $\Q(t)$; equivalently,
$L/\Q(t)$ is finite.  Since $z(t)\in L$, equation \eqref{eq:z-degree} gives
\[
 [L:\Q(t)]\ge3.
\]

The extension is separable.  Each $\Q(t)$-embedding of $L$ into an algebraic
closure sends the matrix entries of the branch to another preimage of the same
output curve.  The resulting points are distinct because those entries
generate $L$.  Every denominator that occurs in $F_2$ is a nonzero element of
$L$, so it remains nonzero under each embedding.  The Jacobian determinant on
the original branch is $1+O(t)$ and hence is nonzero in $L$; its images under
the embeddings are also nonzero.  Thus the fiber over \eqref{eq:curve}
contains at least three distinct simple points.  Lemma~\ref{lem:simple-fiber}
gives $\deg(F_2)\ge3$.

For $n>2$, append to every input matrix the scalar block corresponding to its
center:
\[
 A\oplus0I_{n-2},\quad
 (I+B)\oplus I_{n-2},\quad
 (2I+G)\oplus2I_{n-2},\quad
 (3I+D)\oplus3I_{n-2}.
\]
The map $C$ respects direct sums, so the conjugate $2\times2$ preimages produce
at least three distinct preimages of one $n\times n$ block-diagonal output
curve.  On the original amplified branch the input tends at $t=0$ to the
scalar tuple $(0I_n,I_n,2I_n,3I_n)$, so its $n\times n$ Jacobian determinant is
$1+O(t)$ by Proposition~\ref{prop:formal}.  This determinant is a nonzero
element of $L$, and its conjugates are therefore nonzero as well.  The three
amplified points are simple, and Lemma~\ref{lem:simple-fiber} gives
$\deg(F_n)\ge3$.
\end{proof}

The maps $F_n$ are defined over $\Q$, and the degree in
Proposition~\ref{prop:all-degrees} is geometric.  It is therefore unchanged by
extension to any characteristic-zero constant field.

\begin{proof}[Proof of Theorem~\ref{thm:main}]
If $a,b,c,d$ were free rational functions of $p,q,r,s$, those expressions
would define a free rational left inverse of $F_{P_4}$.  By
Lemma~\ref{lem:free-degree}, $F_n$ would then be birational for some $n\ge2$,
so $\deg(F_n)=1$.  This contradicts Proposition~\ref{prop:all-degrees}.
\end{proof}

\begin{corollary}
The adjacent partners $T=C(a,b)$ and $W=C(d,c)$ are not simultaneously free
rational functions of $p,q,r,s$.  The same holds for the second-layer partners
$M,N$.
\end{corollary}

\begin{proof}
Apply Proposition~\ref{prop:partner} and the equivalent second-layer recovery
described after it.
\end{proof}

\section{Extension from \texorpdfstring{$P_4$}{P4} to non-cographs}

We finish by replacing a potentially invalid specialization of a hypothetical
rational inverse with a specialization of the explicit algebraic branch.  The
latter has all denominators controlled at its scalar base point.

For a graph $G$ on $[m]$, let $F_G$ denote the rational map from the $m$ bottom
pseudo-roots to the neighborhood seed $Y_G$.  Its evaluation on $n\times n$
matrices is denoted $F_{G,n}$.

\begin{lemma}[Central directions]\label{lem:central}
If $\lambda$ is central, then
\[
 C(\lambda,x)=x,
 \qquad
 C(x,\lambda)=\lambda
\]
whenever the indicated inverses exist.  Consequently, after bottom
pseudo-roots outside a vertex set $V$ are assigned distinct central scalars,
pseudo-roots in directions belonging to $V$ reduce to those for the induced
graph $G[V]$, while pseudo-roots in outside directions remain their assigned
scalars.
\end{lemma}

\begin{proof}
Both identities follow because $\lambda-x$ commutes with $x$ and with
$\lambda$.  The consequence follows inductively along any sequence of diamond
moves constructing the relevant pseudo-root.
\end{proof}

\begin{lemma}\label{lem:graph-jacobian}
At a bottom tuple of pairwise distinct central scalars, the Jacobian of
$F_{G,n}$ is the identity for every graph $G$ and every matrix size $n$.
Thus $F_{G,n}$ is dominant and generically finite.
\end{lemma}

\begin{proof}
Every neighborhood pseudo-root is obtained by repeatedly applying $C$, and at
each application Lemma~\ref{lem:first-order} preserves the perturbation in the
direction of the second argument.  Therefore the $i$th seed has first-order
term equal to the perturbation of the $i$th bottom pseudo-root and no other
first-order term.
\end{proof}

\begin{theorem}\label{thm:noncographs}
If $G$ contains an induced $P_4$, then its neighborhood seed map $F_G$ has no
free rational inverse.  More precisely,
\[
 \deg(F_{G,n})\ge3\qquad(n\ge2).
\]
\end{theorem}

\begin{proof}
Let $V=\{v_1,v_2,v_3,v_4\}$ induce the path
$v_1-v_2-v_3-v_4$.  Assign the scalar centers $0,1,2,3$ to these four bottom
directions and assign pairwise distinct rational scalar centers, different
from $0,1,2,3$, to all remaining directions.

For $2\times2$ matrices, place the formal preimage
\eqref{eq:preimage} in the four directions in $V$ and keep every outside input
equal to its scalar center.  Lemma~\ref{lem:central} shows that its image under
$F_{G,2}$ is the curve \eqref{eq:curve} in the four directions in $V$ and is
constant in every outside direction.  All defining denominators have
invertible scalar terms at $t=0$.

The field of entries of this branch contains $z(t)$ and therefore has degree at
least three over $\Q(t)$.  Lemma~\ref{lem:graph-jacobian} gives Jacobian
determinant $1+O(t)$ on the original branch, and its field conjugates remain
nonzero.  As in Proposition~\ref{prop:all-degrees}, there are at least three
distinct simple points in one fiber, so
$\deg(F_{G,2})\ge3$.  Direct-sum amplification with the scalar centers gives
the same conclusion for every $n>2$.

If $F_G$ had a free rational inverse, Lemma~\ref{lem:free-degree} would make
some $F_{G,n}$ birational, contradicting the degree bound.
\end{proof}

Since a graph is a cograph exactly when it has no induced $P_4$, combining
Theorem~\ref{thm:noncographs} with the cograph diamond-generation theorem
\cite[Theorem~1.1(ii)]{Ramani2026} gives the full rational classification.

\begin{corollary}[Rational cograph classification]\label{cor:cograph}
For every finite graph $G$ over a characteristic-zero base field,
\[
 Y_G\text{ rationally generates all Boolean-lattice pseudo-roots}
 \quad\Longleftrightarrow\quad
 G\text{ is a cograph}.
\]
\end{corollary}

\begin{proof}
If $G$ is a cograph, the diamond-generation theorem recovers every pseudo-root
from $Y_G$, and diamond moves are rational operations.  If $G$ is not a
cograph, it contains an induced $P_4$, so Theorem~\ref{thm:noncographs} rules
out rational recovery even of all bottom pseudo-roots.
\end{proof}

The matrix-degree proof of Theorem~\ref{thm:noncographs} is essential here.
Merely substituting central scalars into a hypothetical free rational inverse
would not be sufficient: its denominators could vanish identically on the
specialized locus.  By contrast, the explicit formal branch is defined there,
and its simple conjugate points force a degree bound on the unspecialized
generic matrix map.

\appendix

\section{Reproducibility of the exact calculation}

The ancillary script \nolinkurl{p4_matrix_degree_certificate.py} was verified
with Python~3.12.13 and SymPy~1.14.0 \cite{MeurerEtAl2017}.  A matching
environment can be prepared with
\begin{verbatim}
python3 -m pip install sympy==1.14.0
\end{verbatim}
and the certificate is run with
\begin{verbatim}
python3 p4_matrix_degree_certificate.py
\end{verbatim}
It checks the Pauli multiplication rules and reversal action, substitutes the
formal jet \eqref{eq:formal-jet} into all four seed equations, verifies
\eqref{eq:Aexplicit}, constructs the two branch factors $S,N$, checks their
degrees and the resultant identity \eqref{eq:resultant}, and verifies the
primitive irreducible specialization used in Lemma~\ref{lem:irreducible}.
Every failed check raises an explicit exception; the script does not rely on
Python assertions, which can be disabled by optimization flags.

\end{document}